\documentclass{article}

\usepackage{amsmath}
\usepackage{amsthm}
\usepackage{fullpage}
\usepackage{url}
\usepackage{graphicx}
\newtheorem{thm}{Theorem}

\theoremstyle{definition}

\usepackage{graphicx}
\usepackage{tikz}
\usetikzlibrary{calc}
\usetikzlibrary{patterns}
\tikzstyle{mybox} = [draw=black, fill=white,  thick,
    rectangle, inner sep=10pt, inner ysep=20pt]
\tikzstyle{mybox} = [draw=black, fill=white,  thick,
    rectangle, inner sep=2pt, inner ysep=2pt]

\begin{document}

\title{Solving Cubic Equations By the Quadratic Formula}

\author{Bahman Kalantari \\
Department of Computer Science, Rutgers University, NJ\\
kalantari@cs.rutgers.edu
}
\date{}
\maketitle

\begin{abstract}
Let $p(z)$ be a monic cubic complex polynomial with distinct roots and distinct critical points.  We say a critical point has the {\it Voronoi property} if it lies in the
Voronoi cell of a root $\theta$, $V(\theta)$, i.e. the set of points that are closer to $\theta$ than to the other roots.  We prove at least one critical point has the Voronoi property and
characterize the cases when both satisfy this property. It is known that for any $\xi \in V(\theta)$,  the sequence $B_m(\xi) =\xi - p(\xi) d_{m-2}/d_{m-1}$  converges to $\theta$, where $d_m$ satisfies the recurrence $d_m =p'(\xi)d_{m-1}-0.5 p(\xi)p''(\xi)d_{m-2} +p^2(\xi)d_{m-3}$, $d_0 =1, d_{-1}=d_{-2}=0$.
Thus by the Voronoi  property, there is a solution $c$ of $p'(z)=0$ where  $B_m(c)$ converges to a root  of $p(z)$. The speed of convergence is dependent on the ratio of the distances between $c$ and the closest and the second closest roots of $p(z)$. This results in a different algorithm for solving a cubic equation than the classical methods.    We give polynomiography for an example.
\end{abstract}


\section{Introduction} Solving the cubic equation $p(z)=z^3+a_2z^2+a_1z+a_0=0$ with real or complex coefficients has played a
significant role in the history and development of many areas of mathematics, in particular the discovery of the complex numbers, see e.g. Mazur \cite{Maz}. There are numerous articles on solving a cubic equation,
many restricted to the case of real coefficients.  The development of closed formulas for  solutions in terms of radicals of the coefficients is one of the triumphs in algebra,  albeit the algebraic formulas are quite cumbersome.

Historically, solving a cubic equation received much attention from the greatest of past mathematicians, for example Omar Khayy$\acute{{\rm a}}$m's geometric method for solving some special cases, and of course Cardano's solution, see e.g. Stewart \cite{Ian}.  From the point of view of numerical approximation, even the high school quadratic formula is not practical when the solutions require taking square-roots.  Fortunately, approximation of solutions to a quadratic equation,  in real or complex coefficients,  can be achieved very efficiently, e.g. via Newton's method.   Cayley \cite{cay}, is among the first to have examined Newton's method for complex polynomials. In particular, he analyzed Newton's iterations for a quadratic polynomial, describing their behavior, starting from an arbitrary point in the complex plane. He showed that the basin of attraction of each root is its Voronoi cell.

Cayley attempted to characterize the behavior of Newton's method for the roots of unity but was unable to characterize convergence behavior even for cubic roots.  In the case of $z^3-1=0$,  the behavior of Newton's method was not well-understood until the advent of computers.
As we now know, the basins of attraction have fractal boundary, known as Julia set.  Indeed, even solving a cubic polynomial via Newton's method is not completely settled today.  For instance, it is not known if one can characterize points belonging to the Julia set of Newton's function for $z^3-1$.  Using a real number model of computation, Blum et al. \cite{Blum} proved that given an arbitrary input, and a general cubic
polynomial, to test if the corresponding orbit under Newton's iterations would converge to a root, is undecidable. It is also known there are cubic polynomials, such as $z^3-2z+2$, where Newton's iterates fail to converge to a root for inputs in a set of positive measure.  Given a cubic equation in a normalized form, McMullen \cite{McMullen87} described a rational iteration function where the convergence of the fixed point iterates to a root is assured for any input, expect for the Julia set which is a set of measure zero in this case. A much more complicated result in \cite{McMullen87} implies that for polynomials of degree four or higher there is no {\it generally convergent} rational iteration function. Thus the failure of Newton's method is not an exception.

In this article we prove a property of cubic polynomials in relation to its critical points, interesting in its own right, but also giving rise to a new algorithm for solving a cubic equation. We prove, when both the roots and critical points are distinct,  there exists a critical point $c$ that is closer to a root $\theta$  than to the other roots. We also characterize the cases where both critical points have this {\it Voronoi property}.  Using the Voronoi property, we then describe a sequence, computable in terms of $c$, converging to $\theta$.
By the deflation method, the other roots can be easily approximated.

\section{A Voronoi Property of Cubic Polynomials}

 Given a set of points  $S= \{\theta_1, \dots, \theta_n\}$ in the Euclidean plane, we identify them as complex numbers.  The {\it Voronoi cell} of a particular point $\theta$ in $S$, denoted by $V(\theta)$, is the set of all points in the  plane that are closer to $\theta$ than to any other point in $S$. Each Voronoi cell is an open polygonal set, possibly unbounded.  Voronoi cells of  $\theta_i$'s together with their boundaries partition the plane into disjoint sets, known as Voronoi diagram. For a survey on Voronoi diagrams, see \cite{arun91}. We say a critical point $c$ of $p(z)$ has the {\it Voronoi property} is $c \in V(\theta)$, for some root $\theta$ of $p(z)$. In this section we first prove the followng.

\begin{thm} \label{thm1} {\rm (Voronoi Property)} Let $p(z)$ be a cubic complex polynomial having distinct roots and distinct critical points. Then at least one critical point has the Voronoi property.
\end{thm}

\begin{proof} Assume no critical point has the Voronoi property.
It is easy to argue that when the roots of two cubic polynomials form  similar triangles, the relative location of their critical points will remain unchanged. Thus, without loss of generality we assume the roots of $p(z)$ are $1$, $-1$, and $w=a+ib$,  $i=\sqrt{-1}$, where $a \geq 0$ and $b \geq 0$. We assume that one of the critical points  is equidistant to $1$ and $-1$.  We will assume $b >0$ so that the roots of $p(z)$ are not collinear. The proof for the collinear case follows as a limiting case of the general case and will be omitted. With the roots as $-1,1,w$ we have
\begin{equation} \label{eq1}
p(z)=(z^2-1)(z-w)=z^3-wz^2-z+w.
\end{equation}
The critical points are the solutions to $p'(z)=3z^2-2wz-1=0$:
\begin{equation} \label{eq2}
c_1= \frac{1}{3}(w + \sqrt{w^2+3}), \quad c_2 = \frac{1}{3}(w - \sqrt{w^2+3}).
\end{equation}
Neither one can be zero. Note that the critical points are distinct if and only if $w \not =i\sqrt{3}$.
From $p'(c_j)=0$, we get
\begin{equation} \label{eq3}
w=\frac{3c_j^2-1}{2c_j}, \quad j=1,2.
\end{equation}
Since the critical points cannot be zero, the critical points that is equidistant to $-1$ and $1$ is purely imaginary.
This fact and the above implies that $w$ is also purely imaginary. So
$w=ib$.  From the above, we also get $3c_1c_2=-1$. Thus,  if one critical point is purely imaginary, so is the other. We claim $c_1 \in V(w)$. Since $c_1$ is imaginary, $w^2+3=-b^2+3 <0$. Thus $b > \sqrt{3}$. We have $c_1= i(b + \sqrt{b^2-3})/3$, $c_2 = i(b - \sqrt{b^2-3})/3$.  Since $b > \sqrt{3}$, $c_1$ and $c_2$ lie between $0$ and $w$, and $c_1$ lies between $c_2$ and $w$ (See middle triangle in Figure \ref{fig1}).  To prove
$c_1 \in V(w)$, we need to show the distance between $c_1$ and $w$ is less than the distance between $c_1$ and $-1$. Equivalently,  we need to show
\begin{equation}\label{eq4}
\frac{1}{9}(2b - \sqrt{b^2-3})^2  < 1+ \frac{1}{9}(b - \sqrt{b^2-3})^2.
\end{equation}
Simplifying the above, it is equivalent to showing $3+ 2b^2 <b^4$. The roots of the quadratic $q(x)=x^2-2x-3$ are $3,-1$. Hence for $x > 3$, $q(x) >0$.  But if $x=b^2$, $b^2 >3$. Hence the proof.
\end{proof}

In the next theorem we give a stronger result that classifies the cases where both critical points have the Voronoi property.  Without loss of generality we assume the roots and critical points are as in the previous theorem. We exclude the case of $w=a+ib$, $a=0$, $b > \sqrt{3}$ which was proved in the previous theorem.

\begin{thm} \label{thm2} {\rm (Strong Voronoi Property)} Let $p(z)$ be a cubic complex polynomial having distinct roots $\theta_1=-1$, $\theta_2=1$, $\theta_3=w=a+ib$, $a \geq 0$, $b > 0 $. Its critical points are $c_1= \frac{1}{3}(w + \sqrt{w^2+3})$, and $c_2 = \frac{1}{3}(w - \sqrt{w^2+3})$ and satisfy

(i) If  $a >0$, then $c_2 \in V(-1)$.

(ii) If  $a=0$, $b < \sqrt{3}$, then $c_2 \in V(-1)$ and $c_1 \in V(1)$.
\end{thm}

\begin{proof}  Let $w^2+3=s+id$ with $s,d$ reals. Then
\begin{equation} \label{eq5}
s=(a^2-b^2+3), \quad d=2ab.
\end{equation}
Since $d \geq 0$, it can be shown, see e.g. \cite{Rab}, that
$\sqrt{s+id}= A +iB$, where
\begin{equation} \label{eq6}
A=\frac{1}{\sqrt{2}}{\sqrt{\sqrt{s^2+d^2}+s}}, \quad B=\frac{1}{\sqrt{2}}{\sqrt{\sqrt{s^2+d^2}-s}}.\end{equation}
Thus
\begin{equation} \label{eq7}
c_1=\frac{1}{3}{a+ A }+i \frac{1}{3}{b+B},  \quad c_2=\frac{1}{3}{a- A } +i \frac{1}{3}{b-B}.
\end{equation}

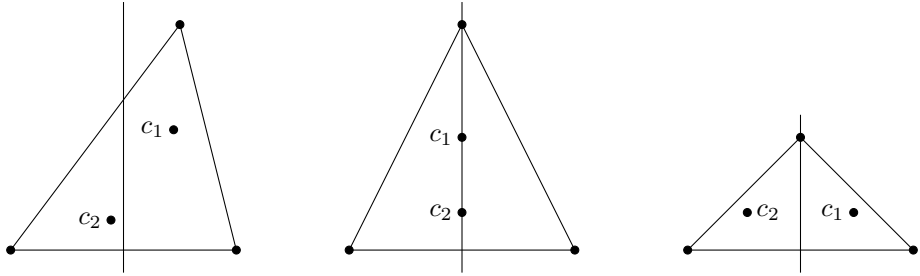
\begin{figure}[htpb]
	\centering
	
	\begin{tikzpicture}[scale=1.5]

\draw (-1.0,0.0) -- (.5,2.0) -- (1.0,0.0) -- cycle;
	%	\draw (-1,0) node[below] {$v_1$};
	%	\draw (.5,2) node[below] {$v_2$};
	%	\draw (1,0) node[above] {$v_3$};
\draw (0.0,-0.2) -- (0,2.2) -- cycle;
\filldraw (-1,0) circle (1pt);
\filldraw (.5,2) circle (1pt);
\filldraw (1,0) circle (1pt);
\filldraw  (.4442,1.0671)  circle (1pt);
\filldraw  (-0.1109,.2662)  circle (1pt);
\draw (.4442,1.0671) node[left] {$c_1$};
\draw  (-0.1109,.2662) node[left] {$c_2$};
%\filldraw (1,.75) circle (2pt);
%\filldraw (1,0) circle (2pt);
%\filldraw (3,2.25) circle (2pt);
%\filldraw (5.5,1.5) circle (2pt);
%\filldraw (1,3) circle (2pt) node[above] {$p$};
%\draw (1,.75) -- (1,0);

\draw (2.0,0.0) -- (3,2.0) -- (4.0,0.0) -- cycle;
	%	\draw (-1,0) node[below] {$v_1$};
	%	\draw (.5,2) node[below] {$v_2$};
	%	\draw (1,0) node[above] {$v_3$};
\draw (3.0,-0.2) -- (3,2.2) -- cycle;
\filldraw (2,0) circle (1pt);
\filldraw (3,2) circle (1pt);
\filldraw (4,0) circle (1pt);
\filldraw  (3,1) circle (1pt);
\filldraw  (3,.3333) circle (1pt);
\draw (3,1) node[left] {$c_1$};
\draw (3,.3333) node[left] {$c_2$};
%\filldraw (1,.75) circle (2pt);
%\filldraw (1,0) circle (2pt);
%\filldraw (3,2.25) circle (2pt);
%\filldraw (5.5,1.5) circle (2pt);
%\filldraw (1,3) circle (2pt) node[above] {$p$};
%\draw (1,.75) -- (1,0);

\draw (5.0,0.0) -- (6,1.0) -- (7,0.0) -- cycle;
	%	\draw (-1,0) node[below] {$v_1$};
	%	\draw (.5,2) node[below] {$v_2$};
	%	\draw (1,0) node[above] {$v_3$};
\draw (6.0,-0.2) -- (6,1.2) -- cycle;
\filldraw (5,0) circle (1pt);
\filldraw (6,1.0) circle (1pt);
\filldraw (7,0) circle (1pt);
\filldraw  (6-.4714,.3333) circle (1pt);
\filldraw  (6.4714,.3333) circle (1pt);
\draw (6-.4714,.3333) node[right] {$c_2$};
\draw  (6.4714,.3333) node[left] {$c_1$};
%\filldraw (1,.75) circle (2pt);
%\filldraw (1,0) circle (2pt);
%\filldraw (3,2.25) circle (2pt);
%\filldraw (5.5,1.5) circle (2pt);
%\filldraw (1,3) circle (2pt) node[above] {$p$};
%\draw (1,.75) -- (1,0);

%\begin{scope}[red]
%          \clip (1.0, .75) -- (1.0,0.0) -- (7.0,0.0) -- (4.0,3.0) -- cycle;
%\fill[color=gray!15] (-10, 10) rectangle (10, -10);
%\fill[color=gray];
%\end{scope}

%%%%%%%
%\draw (0, 10) rectangle (10, 0);	
		
%\begin{scope}[red]
%          \clip (11.0, .75) -- (11.0,0.0) -- (17.0,0.0) -- (14.0,3.0) -- cycle;
%           \clip (11.0, .75) -- (11.0,0.0) -- (17.0,0.0) -- (15.5,1.5) -- (11, 3) --cycle;
%\fill[color=gray!39] (-10, 20) rectangle (20, -10);
%\fill[color=gray];
%\end{scope}

	\end{tikzpicture}
\caption{From left to right: $a >0$;  $a=0$, $b>\sqrt{3}$;  $a=0$, $b <\sqrt{3}$.} \label{fig1}
\label{Fig3Witness}
\end{figure}

Writing the distance between  $c_2$ and $w$ as $d_1$, and the distance between $c_2$ and $-1$ as $d_2$, we
will first claim  that if $b \not = \sqrt{3}$, then $d_1 > d_2$. Computing,
\begin{equation} \label{eq8}
d_1^2= \frac{1}{9}  ({2a} + {A})^2+  \frac{1}{9}  ({2b} + {B} )^2, \quad
d_2^2= \frac{1}{9}  (a - A +3)^2+ \frac{1}{9}  (-{b} + {B})^2.\end{equation}
To verify  $d_1^2 > d_2^2$, after some simplification, is equivalent to verifying
\begin{equation} \label{eq9}
a^2+b^2  + 2aA+ 2bB + 2A >  2a +3.\end{equation}
Suppose $b < \sqrt{3}$. Then,  $s>0$, and thus  $A \geq \sqrt{s}= \sqrt{3+a^2-b^2}$.  Since $bB \geq 0$, in order to prove the above inequality it suffices to argue via calculus that
\begin{equation} \label{eq10}
f(a,b)=b^2  + 2(a+1) \sqrt{3- b^2+a^2} -2a -3 > 0, \quad  0 \leq b < \sqrt{3}.\end{equation}
More precisely, for $a \geq 0$, $f(a,0) > 0$ and $f(a, \sqrt{3}) \geq 0$. Next, $\partial f(a,b)/\partial b=0$  occurs at $a^*=1-b^2/2$.  Since $a \geq 0$, it suffices to show that $g(b)=f(a^*,b)=2b^2-5+(4-2b^2) \sqrt{4-2b^2+b^4/4}$ is positive on $0 \leq b \leq \sqrt{2}$. This is trivial.

Suppose $b > \sqrt{3}$. Then to prove the desired inequality it suffices to show  $2aA + 2bB + 2A \geq 2a.$
If $s \geq 0$, $A \geq a$, and if $s <0$, $B \geq a$, so that $bB \geq a$, hence the proof of claim.

Suppose $a >0$. We claim the real part of $c_2$ is negative. This is valid  if and only if $a < A$.
Substituting  for $A$ from (\ref{eq5}) and  squaring both sides of this inequality and rearranging, we get $2a^2 -s  < \sqrt{s^2+d^2}$. Squaring both sides, simplifying  and substituting for $s$ and $d$,  we get
\begin{equation} \label{eq11}
4a^4 -4a^2(a^2-b^2 +3) < 4a^2b^2.\end{equation}
Simplifying the above gives, $- 12 a^2  b^2 < 0$, which holds true since $b >0$, and $a >0$. Combining this with  we the fact that $d_1 >d_2$, we have proved $c_2 \in V(-1)$.

By the Gauss-Lucas theorem (see e.g. \cite{She}, \cite{Kalbook}, or \cite{KalMon}), the critical points of a polynomial lie in the convex hull of its roots. Applying this to  our case, it follows that both critical points must have nonnegative imaginary part. Hence we have verified the left triangle in Figure \ref{fig1}.

Suppose $a=0$ and $b < \sqrt{3}$. Then $\sqrt{s^2+d^2}+s= |3-b^2|+ 3-b^2$. Thus  the real part of $c_2$ is negative, implying that $c_2$ is closer to $-1$ than to $1$. This together with the fact that $d_1 > d_2$ implies $c_2 \in V(-1)$. By symmetry we have $c_1 \in V(1)$. These with the Gauss-Lucas theorem  justify the right triangle in Figure \ref{fig1}. This completes the proof.
\end{proof}

\section{Approximation of Roots of Polynomials}

In this section we first review  a fundamental family of iteration functions for the approximation of roots  of a
complex polynomial $p(z)$ of degree $n$. We then consider its application in solving a cubic polynomial equation. The {\it basic family} of iteration functions is the collection
\begin{equation} \label{eq12}
B_m(z)=z-p(z) \frac{D_{m-2}(z)} {D_{m-1}(z)}, \quad m=2,3, \dots \end{equation}
where  $D_0(z)=1$,  $D_k(z)=0$ for $k <0$, and, $D_m(z)$ satisfies the recurrence
\begin{equation} \label{eq13}
D_m(z)= \sum_{i=1}^n (-1)^{i-1}p(z)^{i-1}\frac{p^{(i)}(z)}{i!}D_{m-i}(z).\end{equation}

The first two members are, $B_2(z)$ (Newton) and $B_3(z)$ (Halley) iteration functions.  For the rich  history of individual members, their discovery,  many equivalent formulations of the basic family members, and many other properties and applications, see \cite{Kalbook}.  It can be shown that for each fixed $m\geq 2$, there exists a disk centered at a root $\theta$ such that for any $z_0$ in this disk the sequence of fixed point iteration
$z_{k+1}=B_m(z_k)$, $k=0,1,\dots$, is well-defined, and converges to $\theta$. If  $\theta$ is a simple
root,  the order of convergence is $m$.

In contrast to using individual members of the basic family, there is a collective application, using the {\it basic sequence}, $\{B_m(w), m=2, \dots\}$, for some fixed $w$, (see \cite{Kalbook}).
The following theorem describes a pointwise convergence property on each Voronoi cell $V(\theta)$. For a proof of pointwise convergence  see \cite{Kalbook}, and for proof of uniform convergence of the basic family, see \cite{KalDCG}.

\begin{thm} \label{thm3} For any root $\theta$ of $p(z)$, and any $w \in V(\theta)$,
\begin{equation}
\lim_{m \to \infty} B_m(w)=\theta. \qed
\end{equation}
\end{thm}

It can be shown that when the roots of $p(z)$ are simple, the rate of converge is proportional to the ratio $r$ of the distance between $w$ and $\theta$ and the distance between $w$ and the second closest root of $p(z)$. We refer the reader for details to \cite{Kalbook}.  To formally prove this amounts to describing the roots of the characteristic polynomial of the linear homogeneous recurrence relation in (\ref{eq13}), then using well-known representation of the $m$-th term as a linear combination of the $m$-th powers of these roots.  The ratio $D_{m-1}(w)/D_m(w)$ can then be asymptotically estimated to be proportional to $r^m$.

\section{An Algorithm for Solving a Cubic Equation}

Using Theorems \ref{thm1} and \ref{thm3} we describe an algorithm for solving a cubic equation. Given a monic cubic  polynomial $p(z)$,  first compute its critical points. If $p'(z)$  has one solution, then $p(z)=(z^3-a_0)$ and the solutions can be expressed trivially.  Otherwise, let $c$, $c'$ be the two distinct critical points.  Compute the {\it interlaced basic sequence}  $\{B_2(c), B_2(c'), B_3(c), B_3(c'), \cdots,  \}$. Since by Theorem \ref{thm1} either $c$ or $c'$ would fall within the Voronoi cell of one of the roots, by Theorem \ref{thm3} at leat one of the two basic sequences, $\{B_m(c), m=2,3, \dots\}$ or $\{B_m(c'), m=2,3, \dots\}$ will be convergent to a root.  Convergence can be tested by considering the difference between two successive terms in each of the sequences. It should  be noted that according to Theorem \ref{thm3} we do not need to use the exact value of the critical points,  the interlaced basic sequence computed based on approximation to critical points will suffice. Furthermore, any approximation to a root can be improved via a few Newton or Halley iterations.  Once, we have an approximate root $r$, we can use deflation to approximate the other roots, i.e. factor out $(z-r)$ and proceed to approximate the roots of the quadratic quotient.

\section{Polynomiography of An Example}

Consider the polynomial $p(z)=z^3-2z+2$.  The roots are $-1.7693$ and $\pm .88456 +.58974 i$. The triangle of the roots is congruent to the case where the roots are $-1$, $1$ and $w=ib$, $b > \sqrt{3}$.  The critical points are $\pm \sqrt{2/3}$ and $-\sqrt{2/3}$ lies in the Voronoi region of $-1.7693$.

Here we illustrate our proposed algorithm for this polynomial  via {\it polynomiography}. Polynomioraphy is defined as the algorithmic visualization of a polynomial equation via iteration functions, see \cite{Kalbook}.  Newton's method, $z_k= N(z_{k-1})$, where $N(z)=z-p(z)/p'(z)$, fails to converge to a root for inputs in a set of positive measure. The reason is that $N(0)=1$ and $N(1)=0$ and the cycle $\{0,1\}$ is an attractive cycle.  The first image  in Figure \ref{poly} (a {\it fractal polynomiograph}) indicates this, showing the polynomiograph of Newton's method applied to this polynomial. Newton's method fails for any seed in the white area. The second image in Figure \ref{poly} (a {\it non-fractal polynomiograph}) gives the polynomiography under the iterations of the basic sequence.  The dark areas consists of points where the iterations of the basic sequence exceeds a threshold before leading to an approximation of a root.
The critical point $c=-\sqrt{2/3}$ can be shown to lie in the red region in the right-hand side image.

\begin{figure}[ht]
\centering
\includegraphics[width=0.4\textwidth]{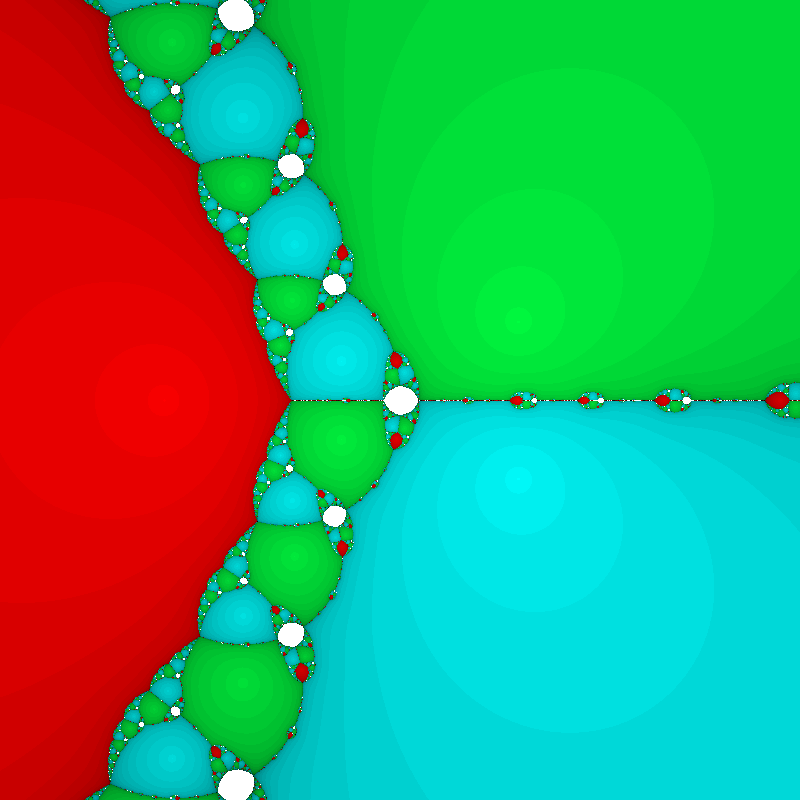}\
\includegraphics[width=0.4\textwidth]{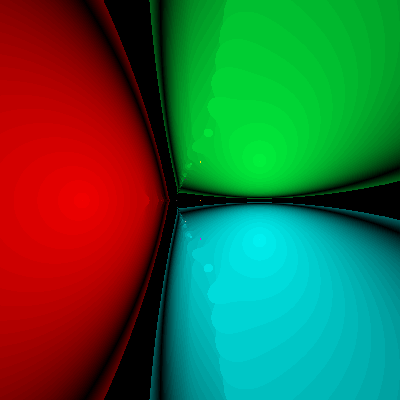}
\caption{Polynomiography of $z^3-2z+2$ using Newton's  method (left) and basic sequence}\label{poly}
\end{figure}

\noindent {\bf Concluding Remarks.}  In this article we have proved a new property of cubic polynomials. Furthermore, we have shown that by  solving a quadratic equation we can generate a sequence that converges to a root of a cubic polynomial.  This in principle justifies the title of the article. While the main result of the article is proving a further property of the ancient problem of root-finding for a cubic equation, the proposed algorithm can also be considered as a practical approach and one that can even be introduced to high school students, making cubic equations and their solutions more tangible. Given one root $\theta$, we can then approximate the other roots via deflation. Any reasonable approximation can be iterated a few times via Newton's iterations to obtain very high accuracy.

Suppose a cubic $p(z)$  has real coefficients. The critical points are either both real, or both complex. If both are complex,  $p(z)$ has one real root.  Then by Theorem \ref{thm2}, case (ii),  each Voronoi cell of a complex root must contain a critical point.  Thus, the algorithm would first find a complex root.  If the critical points are real, $p(z)$ can have three real roots, or one real root and two complex roots. The algorithm would first find a real root.

Generalization of the Voronoi property to polynomials of degree four and higher is an interesting research problem.  We conjecture that with appropriate assumption on distinctness of the roots and critical points
the Voronoi property is valid.  Auckly \cite{Auckly} has shown there is a closed form formula for solving a quartic formula that requires any solution of a related cubic polynomial.  In view of this result and the ones shown here, we can approximate the roots of this cubic equation by solving a quadratic equation, then substitute this into the formulas for solution to a quartic equation in \cite{Auckly}.  If the Voronoi property is valid for quintic polynomials, then in view of the solvability of the general quartic polynomials, and the unsolvability of quintic polynomials via radicals, this together with Theorem \ref{thm3} would give a new algorithm for solving quintic equations.

%% References with bibTeX database:

\bibliographystyle{model1a-num-names}
\bibliography{<your-bib-database>}

{\small

\begin{thebibliography}{9}

\bibitem{Auckly} D. Auckly, Solving quartic with a pencil, \emph{Amer. Math. Monthly} \textbf{114} (2007) 29-39.\filbreak

\bibitem{arun91}  F. Aurenhammer, Voronoi diagrams -- A survey of fundamental geometric data structure,
{\it ACM Computing Surveys}, \textbf{23} (1991) 345--405. \filbreak

 \bibitem{Blum}  L. Blum, F. Cucker, M. Shub, S. Smale, \textit{Complexity and Real Computation}, Springer-Verlag, New York, 1998.\filbreak

\bibitem{cay} A. Cayley, The Newton-Fourier imaginary problem,
{\it American Journal of Mathematics}, \textbf{2} (1879) 97.\filbreak

\bibitem{Kalbook} B. Kalantari, \textit{Polynomial Root-Finding and Polynomiography}, World Scientific,  Hackensack, NJ, 2008.\filbreak

\bibitem{KalDCG} B.  Kalantari,  Polynomial root-finding methods whose basins of attraction approximate Voronoi diagram,  \emph{Discrete \& Computational Geometry},  \textbf{46}  (2011) 187-203.\filbreak

\bibitem{KalMon} B. Kalantari, A geometric modolus principle for polynomials, \emph{Amer. Math. Monthly} \textbf{118} (2011) 931-935.\filbreak

\bibitem{Maz} B. Mazur, \textit{Imagining numbers: particularly the square root of minus fifteen},
 Picador, New York, NY, 2003.\filbreak

\bibitem{McMullen87} C. McMullen, Families of rational maps and iterative root-finding algorithms, {\it The Annals of Math.}, \textbf{125}  (1987) 467-493. \filbreak

\bibitem{Rab} S. Rabinowitz, How to find the square root of a complex number, {\it Mathematics and Informatics Quarterly} \textbf{3}  (1993) 54-56. \filbreak

\bibitem{She} T. Sheil-Small, \textit{Complex Polynomials}, Cambridge
Studies in Advance Mathematics 75, Cambridge University Press, Cambridge,
2002.\filbreak

\bibitem{Ian} I. Stewart, \textit{Why Beauty Is Truth: The History of Symmetry},
Basic Books, New York, NY, 2007.\filbreak

\end{thebibliography}}

%% Authors are advised to submit their bibtex database files. They are
%% requested to list a bibtex style file in the manuscript if they do
%% not want to use model1a-num-names.bst.

%% References without bibTeX database:

% \begin{thebibliography}{00}

%% \bibitem must have the following form:
%%   \bibitem{key}...
%%

% \bibitem{}

% \end{thebibliography}

\end{document}